\documentclass[a4paper,12pt]{amsart}
\usepackage{amsmath,amsthm,amssymb}
\usepackage{mathptmx}
\usepackage{graphicx,graphics}
\usepackage{tikz}
\newtheorem{theorem}{Theorem}[section]

\theoremstyle{definition}

\newtheorem{example}{Example}[section]
\newtheorem{problem}{Problem}[section]
\newtheorem{conjecture}{Conjecture}[section]
\newtheorem{remark}{Remark}[section]
\usepackage{mathrsfs}
\usepackage{enumerate}
\usepackage{float}
\usepackage{graphicx}
\usetikzlibrary{shapes.geometric,arrows}
\usetikzlibrary{mindmap}
\numberwithin{equation}{section}
\begin{document}
    \title{Local antimagic chromatic number of partite graphs }

\author {Pavithra Celeste R}
\address{Pavithra Celeste R, Department of Mathematics, Amrita Vishwa Vidyapeetham, 
 Amritapuri-\textnormal{690525}, India.} 
 \email{pavithracelester@am.amrita.edu} 
 
\author{A V Prajeesh$^{1}$}
\address{A V Prajeesh, Department of Mathematics, Amrita Vishwa Vidyapeetham, 
 Amritapuri-\textnormal{690525}, India.} 
 \email{prajeeshav@am.amrita.edu}

 \author{Sarath V S}
 \address{Sarath V S, Department of Mathematics, Amrita Vishwa Vidyapeetham, 
 Amritapuri-\textnormal{690525}, India. }
\email{writetovssarath@gmail.com}

\subjclass[2010]{Primary 05C78 \\ $^1$Corresponding Author}


\begin{abstract}
 Let $G$ be a connected graph with $|V| = n$ and $|E| = m$. A bijection $f:E\rightarrow \{1,2,...,m\}$ is called a local antimagic labeling of $G$ if for any two adjacent vertices $u$ and $v$, $w(u) \neq w(v)$, where $w(u) = \sum_{e \in E(u)}f(e)$, and $E(u)$ is the set of edges incident to $u$. Thus, any local antimagic labeling induces a proper vertex coloring of $G$ where the vertex $v$ is assigned the color $w(v)$. The local antimagic chromatic number is the minimum number of colors taken over all colorings induced by local antimagic labelings of $G$. Let $m,n > 1$. In this paper, the local antimagic chromatic number of a complete tripartite graph $K_{1,m,n}$, and $r$ copies of a complete bipartite graph $K_{m,n}$ where $m \not \equiv n \bmod 2$ are determined.
\end{abstract}

\thanks{}

\maketitle
\pagestyle{myheadings}
\markboth{Pavithra Celeste R $et.al.$ }{Local antimagic chromatic number of partite graphs}

\section{Introduction}
\par In this paper, we consider only simple and finite graphs. For all graph-theoretic terminology and notation, we refer to Bondy and Murty \cite{MR2368647}.

 \par Antimagic labeling of graphs is one of the oldest graph labeling problems available in the literature. Hartsfield and Ringel introduced the concept of antimagic labeling of graphs in \cite{hartsfield}. Let $G=(V,E)$ be a graph and let $f$ be a bijection from $E \rightarrow \{1,2,...,|E|\}$. For each vertex $u \in V,$ the weight of $u$, $w(u) = \sum_{e \in E(u)}f(e)$, where $E(u)$ is the set of edges incident to $u$. If $w(u) \neq w(v)$ for any two distinct vertices $u\; \& \;v \in V$, then $f$ is called an antimagic labeling of $G$. 

 Hartsfield and Ringel posted the following conjectures in \cite{hartsfield}.

\begin{conjecture} \cite{hartsfield}
 Every connected graph other than $K_2$ is antimagic.
\end{conjecture}
\begin{conjecture} \cite{hartsfield}
  Every tree other than $K_2$ is antimagic.
\end{conjecture}

 In 2017, Premalatha et al. \cite{premalatha} introduced the concept of local antimagic labeling of graphs. Let $G=(V, E)$ be a graph of order $n$ and size $m$ having no isolated vertices. A bijection $f:E \rightarrow \{1,2,..,m\}$ is called a \textit{local antimagic labeling} of $G$ if for all $uv \in E$ we have $w(u) \neq w(v)$ where, $w(u) = \sum_{e \in E(u)}f(e)$. Here $E(u)$ represents all edges incident on the vertex $u$ of $G$. The local antimagic chromatic number $\chi_{la}(G)$ is defined as the minimum number of colors taken over all colorings of $G$ induced by local antimagic labelings of $G$.

One can easily observe that if a graph $G$ is antimagic, it is locally antimagic. Thus the existence of one such connected graph $G$ except $K_2$, which is not locally antimagic, will disprove the antimagic conjecture. 

 Premalatha et al. posted the following conjectures in \cite{premalatha}.
 
\begin{conjecture} \cite{premalatha}
 Every connected graph other than $K_2$ is local antimagic.
\end{conjecture}

\begin{conjecture} \cite{premalatha}
  Every tree other than $K_2$ is local antimagic.
\end{conjecture}

One can also notice that, for every connected graph $G \neq K_2$, $\chi_{la}(G) \geq \chi(G)$, where $\chi(G)$ is the chromatic number of the graph $G$.

The local antimagic chromatic number of a complete bipartite graph $K_{m,n}$ is completely determined in \cite{premalatha,lau}.

\begin{theorem}\cite{lau}
    $$\chi_{l a}(K_{m, n})= \begin{cases}
    n+1 & \text { if } n>m=1 \\
2 & \text { if } n>m \geq 2 \text { and } m \equiv n (\bmod 2). \\
3 & \text { otherwise. }
\end{cases}$$
\end{theorem}

\par Use of combinatorial objects like magic squares, magic rectangles, etc. helps to define labelings much more efficiently and simply.

A \emph{magic square} is an $n\times n$ array whose entries are the consecutive numbers $1,2,\dots,n^2$, each appearing exactly once, such that the sum of each row, column, and the main and main backward diagonal is equal to $\frac{n(n^2 + 1)}{2}.$ For more details, see Kudrle et al. \cite{combdes}.
The following theorem gives the existence of odd by odd magic squares.
\begin{theorem} \cite{FSC}\label{odd}
For $n$ odd, consider the $n\times n$  matrices\\

$N_1=\left[\begin{array}{c c c c c}
1&2&...&n-1&n \\  
2& 3 &...&n&1 \\
.&.&...&.&.\\
.&.&...&.&.\\
.&.&...&.&.\\
n-1&n&...&n-3&n-2\\
n&1&...&n-2&n-1\\
\end{array}\right]$ and \ $N_2 = \left[\begin{array}{c c c c c}
n&n-1&...&2&1 \\  
1& n &...&3&2 \\
.&.&...&.&.\\
.&.&...&.&.\\
.&.&...&.&.\\
n-2&n-3&...&n&n-1\\
n-1&n-2&...&1&n\\
\end{array}\right]$\\

Then an $n\times n$ magic rectangle $M$ is given by $M=N_1+n\left(N_2-J_{nn}\right)$, where $J_{nn}$ is an $n \times n$ matrix with all entries as 1.
\end{theorem}
\begin{remark}\label{oddrk}
 From the proof of Theorem \ref{odd} \cite{FSC}, it can be noted that the result is valid even if we replace the matrix $N_2$ with \\ $N_3 = \left[\begin{array}{c c c c c}
1&n&...&3&2 \\  
2& 1 &...&4&3 \\
.&.&...&.&.\\
.&.&...&.&.\\
.&.&...&.&.\\
n-1&n-2&...&1&n\\
n&n-1&...&2&1\\
\end{array}\right]$,\\ as the matrices $N_1$ and $N_3$ are also orthogonal latin squares.
\end{remark}
\medskip
\par A generalization of a magic square is a \emph{magic rectangle} $MR(a, b)=(m_{ij})$, which is an $a \times b$ array with the entries $1,2,\dots, ab,$ each appearing once, with all its row sums equal to a constant $\rho$ and all its column sums equal to a constant $\sigma$.
\par The sum of the entries in $MR(a,b)$ is $\frac{ab(ab+1)}{2}$ and the magic constants are
\begin{eqnarray*}
	\sigma &=& \sum_{i=1}^a m_{ij}=\frac{a(ab+1)}{2},~~\text{for any}~~j\in \{1,2,\dots,b\}, \  \text{and}\\
	\rho &=& \sum_{j=1}^b m_{ij}=\frac{b(ab+1)}{2},~~\text{for any}~~i\in \{1,2,\dots,a\}.
\end{eqnarray*}

\par  Harmuth \cite{Harmuth1,Harmuth2} proved that such arrays exist whenever $a$ and $b$ are of the same parity, except when, exactly one of $a$ and $b$ is $1$, or $a=b=2$. 
\begin{theorem}\label{MR} \cite{Harmuth1,Harmuth2}
A magic rectangle $MR(a,b)$ exists if and only if $a,b>1$,  $ab>4$ and $a\equiv b\left(\bmod 2\right)$.
\end{theorem}

\par The existence of even by even and odd by odd magic rectangles are also settled in \cite{FSC,Das,Das2,Hega,magic}.

\par A magic rectangle set, $M R S(a, b ; c)$  \cite{Fron2}, is a collection of $c$ arrays $a \times b$ whose entries are elements of $\{1,2, \ldots, a b c\}$, each appearing once, with all row sums in every rectangle equals to a constant $\rho=\frac{b(a b c+1)}{2}$ and all column sums in every rectangle equal to a constant $\sigma=\frac{a(a b c+1)}{2}$.

 In \cite{Fron1}, Froncek proved the existence of ${MRS}(a, b ; c)$.
\begin{theorem}\label{froncek} \cite{Fron1}
    Let $a, b$ and $c$ be positive integers such that $1<a \leq b$. Then a magic rectangle set $M R S(a, b ; c)$ exists if and only if either $a, b, c$ are all odd, or $a$ and $b$ are both even, $c$ is arbitrary, and $(a, b) \neq(2,2)$.
\end{theorem} 

When $a$ and $b$ are of different parity, the magic rectangle $MR(a,b)$ does not exist. Recently, Chai et al. \cite{FSC1} introduced the concept of {nearly magic rectangles}. A \emph{nearly magic rectangle} $NMR(a,b)$ is defined as an $a\times b$ array with $a$ even and $b$ odd that contains each of the integers from the set $\{1,2,\dots,ab\}$ exactly once and the column sums are constant, while the row sums differ by at most one. The following results are from \cite{FSC1}.
\begin{theorem}\cite{FSC1}\label{near0}
Let $b\geq 3$ be odd. Then there exists a $2\times b$ nearly magic rectangle.
\end{theorem}

\begin{theorem}\cite{FSC1}\label{near1}
Let $a\equiv 2 \bmod 4, a\geq 6$, and $b \geq 3$ be odd. Then there exists a nearly magic rectangle of order $a \times b$.
\end{theorem}

\begin{theorem}\cite{FSC1}\label{near2}
Let $a\equiv 0 \bmod 4$ and $b \geq 3$ be odd. Then there exists a nearly magic rectangle of order $a \times b$.
\end{theorem}

\par The concept of Kotzig array was introduced in \cite{kotzig1971magic}. A Kotzig array $KA(a,b)$ is an $a\times b$ array with $a\geq 2$, in which every row contains each of the integers $1,2,...,b$ exactly once and the sum of the entries  in each column is the same constant $c=\frac{a(b+1)}{2}$. The following theorem can be obtained from \cite{MR3013201}.

\begin{theorem}\label{kotzig1} \cite{MR3013201}
Let $h\geq1$. A Kotzig array $KA(2h,b)$ exists for any $b;KA(2h+1,b)$ exists if and only if $b$ is odd.
\end{theorem}

\par  In this article, we determine the local antimagic chromatic number of $r$ copies of  $K_{m,n}$  and  complete tripartite graph $K_{1,m,n}$.

\section{Main results}

Let $K_{m,n}$ be a complete bipartite graph. Alike Magic rectangles, nearly magic rectangles can also be used to define labeling for different graphs efficiently.  One can see easily see that if $\chi_{la}(K_{m,n}) = 2$  then there exists a magic rectangle $MR(m,n)$. In that case, since magic rectangles of different parity do not exist, $\chi_{la}(K_{m,n}) \geq 3$. Now without loss of generality let $m$ be even. From Theorem \ref{near1} and \ref{near2} there exists a nearly magic rectangle $NMR(m,n)$ with $m/2$ row sums as $\frac{n(1+mn)-1}{2}$ and $m/2$ row sums as $\frac{n(1+mn)+1}{2}$. Here the column sum for the $NMR(m,n)$ will be a constant $\frac{m(1+mn)}{2}$.  Let $\{v_1,v_2,\ldots,v_m\}$ and $\{u_1,u_2,\ldots,u_n\}$ be vertices in the first and second partition of $K_{m,n}$ respectively. Now if we label the edges incident to the vertex $v_i$ with the $i^{th}$ row of $NMR(m,n)$, we can observe that $\chi_{la}(K_{m,n}) = 3.$
    
\begin{example} The following is an example of $NMR(4,3).$ 
\begin{center}
$\begin{bmatrix}
 3& 12 &5 \\ 
 10&4&6\\
 2& 9 & 8 \\
 11& 1 &7\\
\end{bmatrix} $   
\end{center}

\par Let $\{v_1,v_2,v_3,v_4\}$ and $\{u_1,u_2,u_3\}$ be the vertices in $K_{4,3}$. Label the edges incident on the vertex $v_i$ with $i^{th}$ row of the above $NMR(4,3).$ We can see that the $w(u_i) = 26$ for all $i$ and $w(v_1)=w(v_2) = 20$ and  $w(v_3)=w(v_4) = 19$. Thus, $\chi_{la}(K_{4,3}) = 3$.
\end{example}
\par Next we move on to characterize the local antimagic chromatic number of $r$ copies of $K_{m,n}$ where $r \geq 2.$  It can be observed, that if an $MRS(m,n;r) $ exists then the $t^{th}$ magic rectangle can be used to label the $t^{th}$ copy of the graph $K_{m,n}$ in a similar way as above. Hence $\chi_{la}(rK_{m,n})$ will be equal to 2. Also, if  $\chi_{la}(rK_{m,n}) = 2$, then there should exists an $MRS(m,n;r)$. Thus whenever $MRS(m,n;r)$ doesn't exist we can conclude that $\chi_{la}(rK_{m,n}) \geq 3$.
Thus using Theorem \ref{froncek}, we get the following result.
\begin{theorem}
    Let $m,n >1$ with $m \neq n $. Then $\chi_{la}(rK_{m,n}) = 2$ whenever $m,n,r$ are all odd or $m,n$ are both even with $(m,n) \neq (2,2).$
\end{theorem}
\begin{theorem}\label{2.2}
    Let $r\geq 2,$ be even and $m,n> 1$ be both odd with $m<n$ and  $m \neq n$, then $\chi_{la}(rK_{m,n}) = 4$. 
\end{theorem}
\begin{proof}
Since $MRS(m,n;r)$ doesn't exist for the case when $r$ is even and $m,n$ both odd, with $m \neq n$, we can conclude that  $\chi_{la}(rK_{m,n}) \geq  3$. But one can see that the sum of the entries $1,2,\ldots,rmn$, is $\frac{rmn(rmn+1)}{2}$ and it's clear that the sum $\frac{rmn(rmn+1)}{2}$ is not divisible by either $rm$ or $rn$. Thus making it impossible to create rectangle sets of order $m \times n$ with either constant row sums or constant column sums. Thus we can conclude that $\chi_{la}(rK_{m,n}) \geq  4$. \par Now if construct $r$ matrices $Z^1,Z^2,\ldots, Z^r$ with row sums, say $\rho_1$ and $\rho_2$, and column sums, say $\sigma_1$ and $\sigma_2$, then we can conclude that $\chi_{la}(rK_{m,n}) = 4$. Next, we construct such $r$ copies of $m \times n$ matrices.
\par It is clear from Theorem \ref{kotzig1} that when $m$ is odd and $n$ is even,  $KA(m,n)$ doesn't exist. In that case, the following relaxed version of Kotzig arrays can be constructed.
\par A $3 \times r$ quasi Kotzig array, $QKA(3,r)$ is defined as follows, where $r=2k, k \geq 1$.
\newline$$
\begin{bmatrix}
    1&2&\ldots&k-1&k&k+1&k+2&\ldots&2k-1&2k\\
    2k-1&2k-3 & \ldots&3&1&2k&2k-2&\ldots&4& 2\\
    k+1&k+2 & \ldots&2k-1&2k&1&2&\ldots&k-1& k
\end{bmatrix}$$
Now if $m>3$ and $m= 2s+1, s \geq 2$, we can attach $s-1$ copy of the matrix,
$$
\begin{bmatrix}
    1&2&\ldots&k&k+1&\ldots&2k-1&2k\\
    2k&2k-1 & \ldots&k+1&k&\ldots&2& 1
\end{bmatrix}$$
to a $QKA(3,r)$ to obtain such $QKA(m,r)$. So that the two different column sums of the $m \times r$ quasi Kotzig array will be $\frac{m(r+1)-1}{2}$ and $\frac{m(r+1)+1}{2}$.
\par Now to construct the matrices $Z^1,Z^2,\ldots,Z^r$, first construct the magic rectangle $MR(m,n)= (m_{i,j})$. Next construct an $m \times n$ matrix $W=(w_{i,j})$ with $w_{i,j} = r(m_{i,j}-1)$. The resultant matrix $W$ will have row sums as $\frac{rn(mn-1)}{2}$ and column sums  as $\frac{rm(mn-1)}{2}$.

\par Next construct $r$ new matrices $U^1,U^2,\ldots,U^r$ which will be added to the matrix $W$ to obtain the required matrices $Z^1,Z^2,\ldots,Z^r$, i.e, $Z^t = (z_{i,j}^t)= U^t+W$. Construct a quasi Kotzig array $QKA(m,r)$ of order $m \times r$. Now the $t^{th}$ column of $QKA(m,r)$ is used as the first column of $U^t$ where $1 \leq t \leq r$. The next $m-1$ columns of $U^t$ are placed with a circulant array constructed from the first column. That is, we have $u_{i ,j}^t=u_{(i+j-1)\bmod m,1}^t$ for $j=2,3, \ldots, m$. When $n >m$, the remaining even-numbered columns are filled with the first column of $U^t$ and all other odd columns are filled as $u_{i,j}^t = r+1-u_{i,j-1}^t$.
\par Here, one can see that for each $U^t$, $1\leq t \leq r/2$ the row sums are  $\frac{nr+n-1}{2}$ and for $U^t$ where $\frac{r}{2}+1\leq t \leq r$, the row sums are going to be $\frac{nr+n+1}{2}$. Also for the first $m$ columns of $U^t$ and all the following even numbered columns the column sum is $\frac{m(r+1)-1}{2}$ and for the remaining columns the sum is  $\frac{m(r+1)+1}{2}$, where $1 \leq t \leq r.$ And for the first $m$ columns of $U^t$ and all the following even-numbered columns the column sum is $\frac{m(r+1)+1}{2}$ and for the remaining columns the sum is  $\frac{m(r+1)-1}{2}$, where $\frac{r}{2}+1 \leq t \leq r.$
\par Now by forming the matrices $Z^t$'s as mentioned above, it is clear that for the matrices $Z^t$'s, where $1 \leq t \leq \frac{r}{2}$, the row sums are  $\rho_1 = \frac{rn(mn)-n-1}{2}$ and when $\frac{r}{2}+1 \leq t \leq r$, the row sums are $\rho_2 = \frac{rn(mn)+n+1}{2}$. Further, one can also see that the column sums of the matrices $Z^{t}$'s are 
$\sigma_1 = \frac{rm(mn)+m-1}{2}$ and $\sigma_2 = \frac{rm(mn)+m+1}{2}$.
\par Since  $m<n$, and $\frac{m-1}{2}<\frac{n-1}{2}<\frac{n+1}{2},$ we have $\sigma_1 
 = \frac{m(mnr)}{2}+\frac{m-1}{2}<\frac{n(mnr)}{2}+\frac{n+1}{2} = \rho_2$ and since $\frac{m+1}{2}<\frac{n+1}{2}$, we have $\sigma_2 
 = \frac{m(mnr)}{2}+\frac{m+1}{2}<\frac{n(mnr)}{2}+\frac{n+1}{2} = \rho_2.$ Also, $\rho_1 - \sigma_1 = (n-m)(\frac{mnr+1}{2})-n \geq 2 (\frac{mnr+1}{2})-n  $, which is always greater than 0. Similarly, $\rho_1 - \sigma_2 = (n-m)(\frac{mnr+1}{2})-n-1  \geq 2 (\frac{mnr+1}{2})-n $, which itself is always greater than 0. Thus all the four entries $\rho_1,\rho_2,\sigma_1$ and $\sigma_2$ are distinct.
 \par Now, to check whether the entries in the matrices are all distinct. We have that for any pairs $(i, j) \neq\left(i^{\prime}, j^{\prime}\right)$, $\left|w_{i,j}-w_{i^{\prime}, j^{\prime}}\right| \geq r$ and $\left|u_{i,j}^s-u_{i^{\prime},j^{\prime}}^t\right| \leq r-1$. Therefore, if $z_{i,j}^s=z_{i^{\prime},j^{\prime}}^t$, we must have $(i, j)=\left(i^{\prime}, j^{\prime}\right)$. Suppose that $z_{i,j}^s=z_{i,j}^t$ and $s \neq t$. Then we have $z_{i,j}^s-z_{i,j}^t=0$ and
$$
z_{i,j}^s-z_{i,j}^t=\left(w_{i,j}+u_{i, j}^s\right)-\left(w_{i,j}+u_{i,j}^t\right)=u_{i,j}^s-u_{i,j}^t=0
$$
Hence $u_{i,j}^s=u_{i,j}^t$. Now, if $j=1$ or $j>m$ and $j$ is even, then $u=u_{i,1}^s=u_{i,1}^t$. But, $u_{i,1}^s$ is an entry of the quasi $\operatorname{Kotzig}$ array $\operatorname{QKA}(m,n)$, namely $q_{i,s}$. Similarly, we have  $u_{i,1}^t=q_{i,t}$. So we have $q_{i,s} = q_{i,t}$, which is not true for a $\operatorname{QKA}(a, b)$, since  $s \neq t$.
Similarly for the case when $j>m$ and $j$ odd, we have $u=r+1-q_{i,s}=r+1-q_{i,t}$ which implies $q_{i,s}=q_{i,t}$, again a  contradiction. Finally, if $2 \leq j \leq m$, then $u=u_{l, 1}^s=u_{l,1}^t$ for some $l\neq i$ and we have $q_{l,s}=q_{l,t}$, a contradiction.
\end{proof}
\begin{example}
    For the graph $4K_{7,11}$ the local antimagic chromatic number is 3. Let $QKA(7,4) = Q$. Here

    \begin{minipage}{9cm}
{
$$MR(7,11) =  \begin{bmatrix}
 77 & 57 & 43 & 56 & 15 & 1 & 64 & 50 & 36 & 22 & 8 \\
 6 & 9 & 34 & 23 & 30 & 37 & 44 & 51 & 58 & 65 & 72 \\
 3 & 10 & 17 & 33 & 31 & 38 & 45 & 54 & 59 & 66 & 73 \\
 39 & 46 & 53 & 60 & 67 & 74 & 4 & 11 & 18 & 25 & 32 \\
 75 & 61 & 47 & 24 & 19 & 5 & 68 & 52 & 40 & 26 & 12 \\
 2 & 20 & 16 & 48 & 62 & 76 & 13 & 27 & 41 & 55 & 69 \\
 71 & 70 & 63 & 29 & 49 & 42 & 35 & 28 & 21 & 14 & 7 
\end{bmatrix}$$}
\end{minipage}
\begin{minipage}{7cm}
$$Q= \begin{bmatrix}
 1 & 2 & 3 & 4 \\
3 & 1 & 4 & 2 \\
 3 & 4 & 1 & 2 \\
1 & 2 & 3 & 4 \\
 4 & 3 & 2 & 1 \\
 1 & 2 & 3 & 4 \\
 4 & 3 & 2 & 1 
\end{bmatrix}$$
\end{minipage} $$ W = \begin{bmatrix}
 304 & 224 & 168 & 220 & 56 & 0 & 252 & 196 & 140 & 84 & 28 \\
20 & 32 & 132 & 88 & 116 & 144 & 172 & 200 & 228 & 256 & 284 \\
8 & 36 & 64 & 128 & 120 & 148 & 176 & 212 & 232 & 260 & 288 \\
152 & 180 & 208 & 236 & 264 & 292 & 12 & 40 & 68 & 96 & 124 \\
296 & 240 & 184 & 92 & 72 & 16 & 268 & 204 & 156 & 100 & 44 \\
 4 & 76 & 60 & 188 & 244 & 300 & 48 & 104 & 160 & 216 & 272 \\
 280 & 276 & 248 & 112 & 192 & 164 & 136 & 108 & 80 & 52 & 24 
 \end{bmatrix}$$
\begin{minipage}{7cm}
{
$$U^1 = \begin{bmatrix}
 1 & 3 & 3 & 1 & 4 & 1 & 4 & 1 & 4 & 1 & 4 \\
 3 & 3 & 1 & 4 & 1 & 4 & 1 & 3 & 2 & 3 & 2 \\
 3 & 1 & 4 & 1 & 4 & 1 & 3 & 3 & 2 & 3 & 2 \\
 1 & 4 & 1 & 4 & 1 & 3 & 3 & 1 & 4 & 1 & 4 \\
 4 & 1 & 4 & 1 & 3 & 3 & 1 & 4 & 1 & 4 & 1 \\
 1 & 4 & 1 & 3 & 3 & 1 & 4 & 1 & 4 & 1 & 4 \\
 4 & 1 & 3 & 3 & 1 & 4 & 1 & 4 & 1 & 4 & 1 
\end{bmatrix}$$}
\end{minipage}
\begin{minipage}{8cm}
$$U^2 = \begin{bmatrix}
 2 & 1 & 4 & 2 & 3 & 2 & 3 & 2 & 3 & 2 & 3 \\
 1 & 4 & 2 & 3 & 2 & 3 & 2 & 1 & 4 & 1 & 4 \\
 4 & 2 & 3 & 2 & 3 & 2 & 1 & 4 & 1 & 4 & 1 \\
 2 & 3 & 2 & 3 & 2 & 1 & 4 & 2 & 3 & 2 & 3 \\
 3 & 2 & 3 & 2 & 1 & 4 & 2 & 3 & 2 & 3 & 2 \\
 2 & 3 & 2 & 1 & 4 & 2 & 3 & 2 & 3 & 2 & 3 \\
 3 & 2 & 1 & 4 & 2 & 3 & 2 & 3 & 2 & 3 & 2 
\end{bmatrix}$$
\end{minipage}
\begin{minipage}{7cm}
{
$$U^3 = \begin{bmatrix}
 3 & 4 & 1 & 3 & 2 & 3 & 2 & 3 & 2 & 3 & 2 \\
 4 & 1 & 3 & 2 & 3 & 2 & 3 & 4 & 1 & 4 & 1 \\
 1 & 3 & 2 & 3 & 2 & 3 & 4 & 1 & 4 & 1 & 4 \\
 3 & 2 & 3 & 2 & 3 & 4 & 1 & 3 & 2 & 3 & 2 \\
 2 & 3 & 2 & 3 & 4 & 1 & 3 & 2 & 3 & 2 & 3 \\
 3 & 2 & 3 & 4 & 1 & 3 & 2 & 3 & 2 & 3 & 2 \\
 2 & 3 & 4 & 1 & 3 & 2 & 3 & 2 & 3 & 2 & 3 
\end{bmatrix}$$}
\end{minipage}
\begin{minipage}{8cm}
$$U^4 = \begin{bmatrix}
 4 & 2 & 2 & 4 & 1 & 4 & 1 & 4 & 1 & 4 & 1 \\
 2 & 2 & 4 & 1 & 4 & 1 & 4 & 2 & 3 & 2 & 3 \\
 2 & 4 & 1 & 4 & 1 & 4 & 2 & 2 & 3 & 2 & 3 \\
 4 & 1 & 4 & 1 & 4 & 2 & 2 & 4 & 1 & 4 & 1 \\
 1 & 4 & 1 & 4 & 2 & 2 & 4 & 1 & 4 & 1 & 4 \\
 4 & 1 & 4 & 2 & 2 & 4 & 1 & 4 & 1 & 4 & 1 \\
 1 & 4 & 2 & 2 & 4 & 1 & 4 & 1 & 4 & 1 & 4 
\end{bmatrix}$$
\end{minipage}

$$Z^1 = \begin{bmatrix}
 305 & 227 & 171 & 221 & 60 & 1 & 256 & 197 & 144 & 85 & 32  \\
 23 & 35 & 133 & 92 & 117 & 148 & 173 & 203 & 230 & 259 & 286  \\
 11 & 37 & 68 & 129 & 124 & 149 & 179 & 215 & 234 & 263 & 290  \\
 153 & 184 & 209 & 240 & 265 & 295 & 15 & 41 & 72 & 97 & 128  \\
 300 & 241 & 188 & 93 & 75 & 19 & 269 & 208 & 157 & 104 & 45  \\
 5 & 80 & 61 & 191 & 247 & 301 & 52 & 105 & 164 & 217 & 276  \\
 284 & 277 & 251 & 115 & 193 & 168 & 137 & 112 & 81 & 56 & 25
\end{bmatrix}$$

$$Z^2  = \begin{bmatrix}
 306 & 225 & 172 & 222 & 59 & 2 & 255 & 198 & 143 & 86 & 31 \\
 21 & 36 & 134 & 91 & 118 & 147 & 174 & 201 & 232 & 257 & 288 \\
 12 & 38 & 67 & 130 & 123 & 150 & 177 & 216 & 233 & 264 & 289 \\
 154 & 183 & 210 & 239 & 266 & 293 & 16 & 42 & 71 & 98 & 127 \\
 299 & 242 & 187 & 94 & 73 & 20 & 270 & 207 & 158 & 103 & 46 \\
 6 & 79 & 62 & 189 & 248 & 302 & 51 & 106 & 163 & 218 & 275 \\
 283 & 278 & 249 & 116 & 194 & 167 & 138 & 111 & 82 & 55 & 26 
\end{bmatrix}$$
$$Z^3 = \begin{bmatrix}
307 & 228 & 169 & 223 & 58 & 3 & 254 & 199 & 142 & 87 & 30 \\
 24 & 33 & 135 & 90 & 119 & 146 & 175 & 204 & 229 & 260 & 285 \\
 9 & 39 & 66 & 131 & 122 & 151 & 180 & 213 & 236 & 261 & 292 \\
155 & 182 & 211 & 238 & 267 & 296 & 13 & 43 & 70 & 99 & 126 \\
 298 & 243 & 186 & 95 & 76 & 17 & 271 & 206 & 159 & 102 & 47 \\
 7 & 78 & 63 & 192 & 245 & 303 & 50 & 107 & 162 & 219 & 274 \\
 282 & 279 & 252 & 113 & 195 & 166 & 139 & 110 & 83 & 54 & 27 
\end{bmatrix}$$

$$Z^4 = \begin{bmatrix}
 308 & 226 & 170 & 224 & 57 & 4 & 253 & 200 & 141 & 88 & 29 \\
 22 & 34 & 136 & 89 & 120 & 145 & 176 & 202 & 231 & 258 & 287 \\
 10 & 40 & 65 & 132 & 121 & 152 & 178 & 214 & 235 & 262 & 291 \\
 156 & 181 & 212 & 237 & 268 & 294 & 14 & 44 & 69 & 100 & 125 \\
 297 & 244 & 185 & 96 & 74 & 18 & 272 & 205 & 160 & 101 & 48 \\
 8 & 77 & 64 & 190 & 246 & 304 & 49 & 108 & 161 & 220 & 273 \\
 281 & 280 & 250 & 114 & 196 & 165 & 140 & 109 & 84 & 53 & 28 
\end{bmatrix}$$
Here $Z^1$ and $Z^2$ matrices have row sum 1699 and $Z^3$ and $Z^4$ matrices have row sum 1700. Also, the column sums are 1081 and 1082.
\end{example}

\begin{theorem}
    Let $m\geq 4$ be even with $r \geq 1$. Then $\chi_{la}(rK_{m,m}) = 3.$
\end{theorem}
\begin{proof}
Clearly, $\chi_{la}(rK_{m,m}) \geq 3$. Because if $\chi_{la}(rK_{m,m}) =2$, then an $MRS(m,m;r)$ will exist with different row sum and column sum, which is impossible.  
\par The existence of $r, m\times m$ matrices, $Z^1,Z^2,\ldots,Z^r$ with entries $1,2,\ldots,rm^2$ for which the row sums are a single constant $\rho_1$ and column sums are either of $\sigma_1$ and $\sigma_2$ will clearly imply the result.
\par To construct such matrices, first construct $MRS(m,2;r)$ using Theorem \ref{froncek} for which the row sums are $\frac{2(2rm+1)}{2}$ and the column sums are $\frac{m(2rm+1)}{2}$.
Similarly, construct $MRS(m,m-2;r)$ and add $2rm$ to every entry in each of the $r$ matrices of $MRS(m,m-2;r)$. Then for the resultant matrices the row sums will be $\frac{(m-2)(rm(m-2)+1)}{2}+2rm^2-4rm$ and the column sums will be $\frac{m(rm(m-2)+1)}{2}+2rm^2$. 
\par Now glue the $i^{th}$ matrix in $MRS(m,2;r)$ with the $i^{th}$ matrix in $MRS(m,m-2;r)$, where $i=1,2,\ldots,r$ to form the required matrices $Z^1,Z^2,\ldots,Z^r$. Hence, the row sums of these matrices are $$\rho_1 = \frac{2(2rm+1)}{2}+\frac{(m-2)(rm(m-2)+1)}{2}+(2rm^2-4rm) = \frac{m(rm^2+1)}{2},$$ and the column sums are $\sigma_1 = \frac{m(2rm+1)}{2}$ and 
$\sigma_2 = \frac{m(rm^2+1)}{2}+rm^2. $
\end{proof}
\begin{theorem}\label{2.4}
    Let $m\geq 3$ and $r\geq 1$ be odd integers, then $\chi_{la}(rK_{m,m}) = 3$.
\end{theorem}
\begin{proof}
    Clearly, since $r$ is odd, $\chi_{la}(rK_{m,m}) \geq 3$. Because if $\chi_{la}(rK_{m,m}) =2$, then an $MRS(m,m;r)$ will exist with different row sum and column sum, which is impossible. Now, to show that  $\chi_{la}(rK_{m,m}) \leq 3$, we proceed as follows.
    \par
 First, construct an $m \times m$ magic square, say $M$, using the De la Loubère method (see \cite{Siam}), which is as follows. Assume $m=2n+1$ for some positive integer $n$.  Fill the $(1,n+1)^{\text{th}}$ entry of $M$ with 1. Once a number is fixed, move one step diagonally up and right to place the next number. If a move leaves $M$, consider the extra cells above the first row to be the same as the last row of $M$ and the cells to the right of the last column of $M$ to be the same as the cells of the first column of $M$. Also, move one step vertically down if a move encounters a filled cell. The matrix $M$ constructed as above is an $m \times m$ magic square with row sum and column sum equal to the magic number $\frac{1}{2}m(m^2+1)$. Also, it is easy to see that the elements of the $(n+1)^{\text{th}}$ column of $M$ form a finite arithmetic progression having the first term 1, common difference $m+1$ and the last term $m^2$.
 \par 
 Now, the magic square $M=(m_{i,j})$ is modified to get another matrix, say $M^* = (m^*_{i,j})$  as, \begin{equation*}
 m^*_{i,j}=\left\{ \begin{array}{ll}
      m_{i,j}& \text{if}\; j\neq n+1,  \\
     m_{i-1,j} & \text{if}\; j=n+1,\;\& \; i\neq 1,\\
     m_{n+1,n+1}& \text{if}\; j=n+1,\;\& \; i= 1.
 \end{array}\right.
 \end{equation*}
The above modification does not make any changes to the column sums, but the row sums will change to $\frac{1}{2}m(m^2+1)+m^2-1$ for the first row and to $\frac{1}{2}m(m^2+1)-m-1$, for all other rows.
\par Next step is to construct a set of $r$ matrices, say $Z^1, Z^2,\dots, Z^r$ having entries $1,2,3,\dots, $ $rm^2$, each appearing once, with each matrix having all its column sums equal to $\sigma=\dfrac{m(rm^2+1)}{2}$ and the row sums equal to $\rho_1=\dfrac{m(rm^2+1)}{2}+r(m^2-1)$ for the first row and $\rho_2=\dfrac{m(rm^2+1)}{2}-r(m+1)$ for all other rows. For that, a Kotzig array $KA(m,r)$ having the row sums $\dfrac{r(r+1)}{2}$ and column sums $\dfrac{m(r+1)}{2}$ is considered. Then, $r$ matrices, namely $U^1, U^2,\dots, U^r$, are obtained in such a way that the first column of $U^t = (u^t_{i,j}),t=1,2,\dots, r$, is the same as the $t^{\text{th}}$ column of $KA(m,r)$. The next $m-1$ columns are obtained iteratively, such that $u^t_{i,j}=u^t_{i+1,j-1}, i=1,2,3,\dots,m-1; j=2,3,\dots,m$ and $u^t_{i,j}=u^t_{1,j-1}, i=m; j=2,3,\dots,m$. Hence, each matrix $U^t,t=1,2,\dots, r$ have both row sums and column sums equal to $\dfrac{m(r+1)}{2}$. Now, the matrices $Z^1, Z^2,\dots, Z^r$ having the desired row and column sums are obtained such that $Z^t=U^t+W$ for $t=1,2,\dots, r,$ where $W=(w_{i,j})$ with $w_{i,j} = r(m^*_{i,j}-1)$.
\par 
Now, what remains is to show that the entries of the matrices $Z^1, Z^2,\dots Z^r$ consist of all the integers $1,2,\dots,rm^2$, such that each of these integers appears exactly once. First, consider two unequal pairs, say $(i,j)\;\text{and}\; (i',j')$. Then $z^s_{i,j}=z^t_{i',j'}$ if and only if $w_{i,j}-w_{i',j'}=u^t_{i',j'}-u^s_{i,j}$, which is impossible since it is obvious that $|w_{i,j}-w_{i',j'}|\geq r$ and $|u^t_{i',j'}-u^s_{i,j}|\leq r-1$. Hence, if $(i,j)\neq (i',j')$ then $z^s_{i,j}\neq z^t_{i',j'}$. Now, consider the numbers $z^s_{i,j}\; \text{and}\;z^t_{i,j}$, where $s\neq t$. Then,   $z^s_{i,j}=z^t_{i,j}$ if and only if $u^s_{i,j}=u^t_{i,j}$. Since these entries are in the same row and column, they will be in the same row of the Kotzig array $KA(m,r)$. Hence, $u^s_{i,j}$ can not be equal to $u^t_{i,j}$ as there are no identical entries in a row of a Kotzig array.
\end{proof}
\begin{theorem}
    Let $m\geq 3$ be odd, and  $r\geq 2$ be even. Then $3\leq \chi_{la}(rK_{m,m})\leq 6$.
\end{theorem}
\begin{proof}
    Using a similar proof technique as in Theorem \ref{2.4} by using a $QKA(m,r)$ (Construction given in Theorem \ref{2.2}) instead of $KA(m,r)$, it can be proved that $\chi_{la}(rK_{m,m})\leq 6.$
\end{proof}
\begin{example}
    Consider the graph $4K_{7,7}$. Here the matrices $M^*,W , U^t$'s and $Z^t$'s are given for $t = 1,2,3,4.$ \vskip 1 cm
    \begin{minipage}{7cm}
{
$$  M^* = \begin{bmatrix}
 30 & 39 & 48 & 9 & 10 & 19 & 28 \\
 38 & 47 & 7 & 17 & 18 & 27 & 29 \\
 46 & 6 & 8 & 25 & 26 & 35 & 37 \\
 5 & 14 & 16 & 33 & 34 & 36 & 45 \\
 13 & 15 & 24 & 41 & 42 & 44 & 4 \\
 21 & 23 & 32 & 49 & 43 & 3 & 12 \\
 22 & 31 & 40 & 1 & 2 & 11 & 20 \\

\end{bmatrix}$$}
\end{minipage}
\begin{minipage}{7cm}
$$QKA(7,4)= \begin{bmatrix}
 1 & 2 & 3 & 4 \\
3 & 1 & 4 & 2 \\
 3 & 4 & 1 & 2 \\
1 & 2 & 3 & 4 \\
 4 & 3 & 2 & 1 \\
 1 & 2 & 3 & 4 \\
 4 & 3 & 2 & 1 \\
\end{bmatrix}$$
\end{minipage}
\vskip 1 cm
 $  W = \begin{bmatrix}
 116 & 152 & 188 & 32 & 36 & 72 & 108 \\
148 & 184 & 24 & 64 & 68 & 104 & 112 \\
 180 & 20 & 28 & 96 & 100 & 136 & 144 \\
 16 & 52 & 60 & 128 & 132 & 140 & 176 \\
 48 & 56 & 92 & 160 & 164 & 172 & 12 \\
 80 & 88 & 124 & 192 & 168 & 8 & 44 \\
 84 & 120 & 156 & 0 & 4 & 40 & 76 \\
\end{bmatrix}$
\vskip 0.5 cm
\begin{minipage}{7cm}
{
$$U^1 = \begin{bmatrix}
 1 & 3 & 3 & 1 & 4 & 1 & 4  \\
 3 & 3 & 1 & 4 & 1 & 4 & 1  \\
 3 & 1 & 4 & 1 & 4 & 1 & 3  \\
 1 & 4 & 1 & 4 & 1 & 3 & 3  \\
 4 & 1 & 4 & 1 & 3 & 3 & 1  \\
 1 & 4 & 1 & 3 & 3 & 1 & 4 \\
 4 & 1 & 3 & 3 & 1 & 4 & 1 
\end{bmatrix}$$}
\end{minipage}
\begin{minipage}{8cm}
$$U^2 = \begin{bmatrix}
 2 & 1 & 4 & 2 & 3 & 2 & 3  \\
 1 & 4 & 2 & 3 & 2 & 3 & 2 \\
 4 & 2 & 3 & 2 & 3 & 2 & 1 \\
 2 & 3 & 2 & 3 & 2 & 1 & 4  \\
 3 & 2 & 3 & 2 & 1 & 4 & 2  \\
 2 & 3 & 2 & 1 & 4 & 2 & 3 \\
 3 & 2 & 1 & 4 & 2 & 3 & 2 
\end{bmatrix}$$
\end{minipage}
\vskip 1 cm
\begin{minipage}{7cm}
{
$$U^3 = \begin{bmatrix}
 3 & 4 & 1 & 3 & 2 & 3 & 2  \\
 4 & 1 & 3 & 2 & 3 & 2 & 3 \\
 1 & 3 & 2 & 3 & 2 & 3 & 4 \\
 3 & 2 & 3 & 2 & 3 & 4 & 1 \\
 2 & 3 & 2 & 3 & 4 & 1 & 3\\
 3 & 2 & 3 & 4 & 1 & 3 & 2  \\
 2 & 3 & 4 & 1 & 3 & 2 & 3  
\end{bmatrix}$$}
\end{minipage}
\begin{minipage}{8cm}
$$U^4 = \begin{bmatrix}
 4 & 2 & 2 & 4 & 1 & 4 & 1  \\
 2 & 2 & 4 & 1 & 4 & 1 & 4 \\
 2 & 4 & 1 & 4 & 1 & 4 & 2  \\
 4 & 1 & 4 & 1 & 4 & 2 & 2 \\
 1 & 4 & 1 & 4 & 2 & 2 & 4  \\
 4 & 1 & 4 & 2 & 2 & 4 & 1  \\
 1 & 4 & 2 & 2 & 4 & 1 & 4 
\end{bmatrix}$$
\end{minipage}

$$Z^1 = \begin{bmatrix}
 117 & 155 & 191 & 33 & 40 & 73 & 112 \\
151 & 187 & 25 & 68 & 69 & 108 & 113 \\
 183 & 21 & 32 & 97 & 104 & 137 & 147 \\
17 & 56 & 61 & 132 & 133 & 143 & 179 \\
 52 & 57 & 96 & 161 & 167 & 175 & 13 \\
81 & 92 & 125 & 195 & 171 & 9 & 48 \\
\end{bmatrix}$$

$$Z^2  = \begin{bmatrix}
 118 & 153 & 192 & 34 & 39 & 74 & 111 \\
 149 & 188 & 26 & 67 & 70 & 107 & 114 \\
 184 & 22 & 31 & 98 & 103 & 138 & 145 \\
 18 & 55 & 62 & 131 & 134 & 141 & 180 \\
 51 & 58 & 95 & 162 & 165 & 176 & 14 \\
 82 & 91 & 126 & 193 & 172 & 10 & 47 \\
 87 & 122 & 157 & 4 & 6 & 43 & 78 \\
\end{bmatrix}$$
$$Z^3 = \begin{bmatrix}
 119 & 156 & 189 & 35 & 38 & 75 & 110 \\
152 & 185 & 27 & 66 & 71 & 106 & 115 \\
 181 & 23 & 30 & 99 & 102 & 139 & 148 \\
 19 & 54 & 63 & 130 & 135 & 144 & 177 \\
 50 & 59 & 94 & 163 & 168 & 173 & 15 \\
 83 & 90 & 127 & 196 & 169 & 11 & 46 \\
 86 & 123 & 160 & 1 & 7 & 42 & 79 \\
\end{bmatrix}$$

$$Z^4 = \begin{bmatrix}
 120 & 154 & 190 & 36 & 37 & 76 & 109 \\
150 & 186 & 28 & 65 & 72 & 105 & 116 \\
 182 & 24 & 29 & 100 & 101 & 140 & 146 \\
 20 & 53 & 64 & 129 & 136 & 142 & 178 \\
 49 & 60 & 93 & 164 & 166 & 174 & 16 \\
 84 & 89 & 128 & 194 & 170 & 12 & 45 \\
 85 & 124 & 158 & 2 & 8 & 41 & 80 \\
\end{bmatrix}$$
Here, the matrix $M^*$ has row sums of 183 for the first six rows and 127 for the last row. Also, the column sums are 175. The column sums of $Z^1$ and $Z^2$ is 689 and for $Z^3$ and $Z^4$ is 690. The row sum of the first 6 columns of $Z^1$ and $Z^2$ is 721 and for $Z^3$ and $Z^4$ is 722. Also, the $7^{th}$ row sum of $Z^1$ and $Z^2$ is 497 and for $Z^3$ and $Z^4$ is 498. Thus $\chi_{la}(4K_{7,7}) = 6.$
\end{example}
Next, we characterize the local antimagic chromatic number of a complete tri-partite graph $K_{1,m,n}$. The local antimagic chromatic number of $K_{1,2,n}$ has already been discussed in \cite{lau}. 
\par The following terminology is used in the upcoming Theorems.  Assume that the graph $K_{1,m,n}$ has a vertex $x$ in first partition $X$, $m$ vertices $v
_1,v_2,\ldots,v_m$ in the second partition $U$ and $n$ vertices  $u_1,u_2,\ldots,u_n$ in third partition $W$ of the vertex set of the graph $K_{1,m,n}$. 

The procedure for labeling the graphs $K_{1,m,n}$ in the following theorems is given below. The same is followed in Theorem \ref{partite1}, \ref{partite2}, \ref{partite3}, and \ref{partite4}.

In each case, an $(m+1)\times (n+1)$ matrix $B=(b_{i,j})$ is constructed in such a way that the entry in the (1,1) position of the matrix is left blank. Then to label the edges $v_ix$ we use the entries $b_{i+1,1}$ from $B$ and to label the edges $u_jx$, we use the entries $b_{1,j+1}$ of the matrix $B.$ Then to label the edge $v_iu_j$, we use the entry $b_{i+1j+1}$, where  $1\leq i \leq m$ and $1 \leq j \leq n.$
\begin{theorem}\label{partite1}
Let  $m,n \geq 2$ with  $m \equiv n \bmod 2 $ and $ m \neq n$. Then $\chi_{la}(K_{1,m,n})=3.$
\end{theorem}
\begin{proof}
     For the graph $K_{1,m,n}$, it is clear that  $\chi(K_{1,m,n}) = 3$ and since $\chi_{la}(K_{1,m,n}) \geq \chi(K_{1,m,n}),$ we have  $\chi_{la}(K_{1,m,n}) \geq 3.$ 
    A new matrix $A=(a_{i,j})$ of order
    $(m+1)\times(n+1)$ is constructed using $MR(m+1,n+1)$, the existence of which is clear from Theorem \ref{MR}. We swap the columns and rows of the matrix $MR(m+1,n+1)$ to bring the entry 1 in $(1,1)^{th}$ position of the matrix and subtract 1 from every entry in the matrix $A$. Then, the $(1,1)^{th}$ position of the matrix is left blank. The resultant matrix so obtained is the required $B$ matrix. 
Thus $\chi_{la}(K_{1,m,n}) = 3.$ 
\end{proof}

\begin{theorem}\label{partite2}
Let  $m,n \geq 2$ with  $m \not \equiv n \bmod 2 $. Then $3\leq \chi_{la}(K_{1,m,n})\leq 4$.
\end{theorem}
\begin{proof}
   The proof is similar to that of Theorem \ref{partite1}. Clearly, $\chi_{la}(K_{1,m,n}) \geq 3$. Now, to label the graph $K_{1,m,n}$ we use an $NMR(m+1,n+1)$. Without loss of generality assume that $m$ is odd and $n$ is even.
   \par As in Theorem \ref{partite1}, a matrix $B$ of order
    $(m+1)\times(n+1)$ is constructed using $NMR(m+1,n+1)$, the existence of which is clear from Theorem \ref{near0},\ref{near1} and \ref{near2}. We swap the columns and rows of the matrix $NMR(m+1,n+1)$ to bring the entry 1 in $(1,1)^{th}$ position of the matrix and subtract 1 from every entry in the matrix. Then, the $(1,1)^{th}$ position of the matrix is left blank. The resultant matrix so obtained is the required $B$ matrix. Labeling is done in a similar fashion as in Theorem \ref{partite1}. Thus one can see that $\chi_{la}(K_{1,m,n})\leq 4.$
\end{proof}
\begin{theorem}\label{partite3}
    Let  $n\geq 3$ be odd. Then $\chi_{la}(K_{1,n,n})=3$.
\end{theorem}
\begin{proof}
    To label the graph $K_{1,n,n}$, construct an $(n+1) \times (n+1)$ matrix $A=(a_{i,j})$ as follows.
    \begin{center}$
    \begin{bmatrix}
        1&2n+2&2n+3&4n+4&...&n^2&(n+1)^2\\
        2&2n+1&2n+4&4n+3&...&n^2+1&(n+1)^2-1\\
        .&.&.&.&...&.&.\\
        .&.&.&.&...&.&.\\
        .&.&.&.&...&.&.\\
        n& n+3&3n+2&3n+5&...&n^2+n-1&(n+1)^2-(n-1)\\
        n+1&n+2&3n+3&3n+4&...&n^2+n&(n+1)^2-n
    \end{bmatrix}$
    \end{center}
    One can also see that the elements in the $i^{th}$ row of the matrix $A$ is of form $i,4s+1-i,4s+i,8s+1-i,8s+i,...,4s(s-1)+i,4s^2+1-i$, where $s=\frac{n+1}{2}$ and $1 \leq i \leq n+1$. Here the row sums are all a single constant ${\frac{(n+1)}{2}\left[\frac{4n^2+8n+11}{2}\right]}$.
    \par Subtract one from each entry of the matrix $A$. Now the $k^{th}$ column sum of the updated matrix is $(n+1)\frac{(n+2)+2(k-1)(n+1)-2}{2}$ where $1\leq k\leq n+1$. Now the sum of the $(\frac{n+3}{2})^{th}$ column entries is $\left(\frac{n+1}{2}\right)\left(n^2+3n+1\right)$. Further, do $s$ swaps for each pair of columns mainly $j^{th}$ and $(n+3-j)^{th}$ one, where $2\leq j\leq \frac{n+1}{2}$.
    The swapping is done in such a way that the entries $a_{i,j}$'s are swapped with $a_{i,n+3-j}$,(total $s$ of such swaps).  Here the difference between the $j^{th}$ and $(n+3-j)^{th}$ entries in $i^{th}$ row is $4s\left(s-j+1\right)$. Thus  $4s^2\left(s-j+1\right)$ is added to the $j^{th}$ column sum and $4s^2\left(s-j+1\right)$ is subtracted from $(n+3-j)^{th}$ column sum. That is,
    \begin{eqnarray}
    \frac{(n+1)\left[(n+2)+2(j-1)(n+1)-2\right]}{2} + 4s^2\left(n-j+1\right)
    \end{eqnarray}
    and 
    \begin{eqnarray}
    \frac{(n+1)\left[(n+2)+2(n-j+2)(n+1)-2\right]}{2}-4s^2\left(n-j+1\right)
    \end{eqnarray}
Both of which are equal to $\left(\frac{n+1}{2}\right)\left(n^2+3n+1\right)$ the required column sum. The resultant matrix is the required $B$ matrix which can be used to label the graph $K_{1,n,n}$. Thus we get $\chi_{la}(K_{1,n,n})=3$.
\end{proof}
\begin{theorem}\label{partite4}
    Let $m\geq 2$ be an even. Then $3\leq 
    \chi_{la}(K_{1,m,m})\leq 4$. 
\end{theorem}
\begin{proof}
The lower bound is obvious from the fact that the $\chi(K_{1,m,m})=3$. Now to get the upper bound, one can proceed as follows.
    \par First construct a magic square $A=(a_{i,j})$ of odd order $(m+1)\times (m+1)$, as in Remark \ref{oddrk} by taking $m+1 = p$. Then, the row sum and column sum of $A$ are $\frac{p(p^2+1)}{2}$. Now, swap the first entry of $j^{\text{th}} $ column with the first entry of $(j+1)^{\text{th}}$ column for $j = 2,4,6,...,p-1$. It can be seen that the difference between these swapped elements is $p-1$. Thus all the $j$ columns, where $j = 2,4,6,..,p-1,$ have a column sum of  $\frac{p(p^2+1)}{2}-(p-1)$ and all the remaining columns except the first column have a sum of $\frac{p(p^2+1)}{2}+(p-1)$. It can be noted that the row sum is unaffected by these swappings. Now, subtract one from each matrix element for a new matrix $B$, having its $(1,1)^{\text{th}}$ entry 0.
    \par Observe that the matrix $B$ has equal row sums $\frac{p(p^2+1)}{2}-p$, and three different column sums $\frac{p(p^2+1)}{2}-p-(p-1)$, $\frac{p(p^2+1)}{2}-p+(p-1)$ and $\frac{p(p^2+1)}{2}-p$ with $\frac{p(p^2+1)}{2}-p$ being the column sum of the first column alone. Thus $B $ can be used to label the graph $K_{1,m,m}$, which gives $\chi_{la}(K_{1,m,m})\leq 4$.
\end{proof}
\section{Conclusion and scope}
In this paper, the local antimagic chromatic number of a complete tripartite graph $K_{1,m,n}$ is determined. Further, the local antimagic chromatic number of $r$ copies of a complete bipartite graph $K_{m,n}$ is also determined, except for the case when $m$ and $n$ are of different parity. Thus the following problem is of interest.
\begin{problem}
    Determine the local antimagic chromatic number of $rK_{m,n}$ where $r \geq 2$ and $m \not \equiv n \bmod 2$.
\end{problem}

\begin{thebibliography}{}
	
\bibitem{premalatha}
 	 S. Arumugam, K. Premalatha, M. Bača and A. Semaničová-Feňovčíková, Local antimagic vertex coloring of a graph. Graphs and combinatorics, 33(2), 275-285, 2017.

    
\bibitem{MR2368647} 
    J.A. Bondy, and U.S.R. Murty, Graph Theory, Springer, New York, 2008.     
    
\bibitem{FSC}
    F.S. Chai, A. Das and C. Midha, Construction of magic rectangles of odd order, Australasian Journal of Combinatorics, { 55}, 131--144, 2013.
	
\bibitem{FSC1}
	F.S. Chai, R. Singh, J. Stufken, Nearly magic rectangles, Journal of Combinatorial Designs, { 27}, 562--577, 2019.    

    
\bibitem{Das2}
	J.P. De Los Reyes, A. Das, C. K. Midha and P. Vellaisamy, On a method to construct magic rectangles of even order, Utilitas Mathematica, {80}, 277--284, 2009.
	
\bibitem{Das}
	J.P. De Los Reyes, A. Das, and C.K. Midha, A matrix approach to construct magic rectangles of even order, Australasian Journal of Combinatorics, {40}, 293--300, 2008.
			


\bibitem{Fron1}
  D. Fron\v{c}ek, Magic rectangle sets of odd order, Australasian Journal of Combinatorics, 67, 345–351, 2017.

\bibitem{Fron2}
 D. Fron\v{c}ek, Handicap distance antimagic graphs and incomplete tournaments, AKCE
International Journal of Graphs and Combinatorics, 10(2), 119–127, 2013. 

 
\bibitem{Hega}
	T. Hagedorn, Magic rectangles revisited, Discrete mathematics,  {207}, 65--72, 1999.
	
\bibitem{Harmuth1}
	T. Harmuth, Ueber magische Quadrate und \"ahnliche Zahlenfiguren, Archiv der Mathematik und Physik, {66}, 286--313, 1881.
	
\bibitem{Harmuth2}
	T. Harmuth, Ueber magische Rechtecke mit ungeraden Seitenzahlen, Archiv der Mathematik und Physik, {66}, 413--447, 1881.
	
\bibitem{hartsfield}
    N. Hartsfield and G. Ringel, Pearls in graph theory: a comprehensive introduction, Courier Corporation, 2013.	

    

\bibitem{kotzig1971magic}
	A. Kotzig. On magic valuations of trichromatic graphs. Reports of the CRM, 1971. 
	
\bibitem{combdes} 
	J.M. Kudrle, and S.B. Menard, Magic Squares, The CRC handbook of combinatorial designs (C. J. Colbourn,
	and J. H. Dinitz, eds.), 2nd ed., CRC Press, Boca Raton, FL, 524--527, 2007.
		

\bibitem{lau}    
    G. C. Lau, H. K. Ng, and  W. C. Shiu, Affirmative solutions on local antimagic chromatic number. Graphs and combinatorics, 36(5), 1337-1354, 2020.
    
\bibitem{magic}
   C K. Midha, J.P. De Los Reyes, A. Das, and L.Y. Chan, On a method to construct magic rectangles of odd order. Statistics and Applications, 6(1), 17-24,2008.    
	
\bibitem{MR3013201}
   A.M. Marr and W.D. Wallis. Magic Graphs. Birkhäuser/Springer, New York, 2013.

\bibitem{Arumugam}
    K.  Premalatha, S. Arumugam,  Yi-Chun Lee, and Tao-Ming Wang,  Local antimagic chromatic number of trees-I. Journal of discrete mathematical sciences and cryptography, 25(6), 1-12, 2022.

  \bibitem{Siam}
 M. Kraitchik, Magic Squares. Mathematical Recreations. Norton, New York, 1942.


\end{thebibliography}

\end{document}